\documentclass[reqno]{amsart}

\usepackage{amscd,amssymb,amsmath,graphicx,verbatim}
\newtheorem {thm}{Theorem}[section]
\newtheorem{lem}[thm]{Lemma}
\newtheorem{prop}[thm]{Proposition}
\newtheorem{cor}[thm]{Corollary}
\newtheorem{df}[thm]{Definition}

\newtheorem{ex}[thm]{Example}
\newtheorem{exs}[thm]{Examples}

\begin{document}

\title[Reflexivity of Rings via Nilpotent Elements]{Reflexivity of Rings via Nilpotent Elements}
\author{A. Harmanci}
\address{Abdullah Harmanci, Department of Mathematics, Hacettepe
University, Ankara,~ Turkey}\email{harmanci@hacettepe.edu.tr}
\author{H. Kose}
\address{Handan Kose, Department of Mathematics, Ahi Evran University, Kirsehir, Turkey}
\email{handan.kose@ahievran.edu.tr}
\author{Y. Kurtulmaz}
\address{Yosum Kurtulmaz, Department of Mathematics, Bilkent University,
Ankara, Turkey}\email{yosum@fen.bilkent.edu.tr}
\author{B. Ungor}
\address{Burcu Ungor, Department of Mathematics, Ankara
University, Ankara, Turkey}\email{bungor@science.ankara.edu.tr}

\date{}

\begin{abstract} An ideal $I$ of a ring $R$ is called {\it left N-reflexive} if for any
$a\in$ nil$(R)$, $b\in R$, being $aRb \subseteq I$ implies $bRa
\subseteq I$ where  nil$(R)$ is the set of all nilpotent elements
of $R$. The ring $R$ is called {\it left N-reflexive} if the zero
ideal is left N-reflexive. We study the properties of left
N-reflexive rings and related concepts. Since reflexive rings and
reduced rings are left N-reflexive, we investigate the sufficient
conditions for left N-reflexive rings to be reflexive and reduced.
We first consider basic extensions of left N-reflexive rings. For
an ideal-symmetric ideal $I$ of a ring $R$, $R/I$ is left
N-reflexive. If an ideal $I$ of a ring $R$ is reduced as a ring
without identity and  $R/I$ is left N-reflexive, then $R$ is left
N-reflexive. If $R$ is a quasi-Armendariz ring and the
coefficients of any nilpotent polynomial in $R[x]$ are nilpotent
in $R$, it is proved that $R$ is left N-reflexive if and only if
$R[x]$ is left N-reflexive. We show that the concept of
N-reflexivity is weaker than that of reflexivity and stronger than
that of left N-right idempotent reflexivity and right idempotent
reflexivity which are introduced in Section 5.

\vspace{2mm}



\vskip 0.5cm
\end{abstract}
\maketitle
\section{Introduction}
Throughout this paper, all rings are associative with identity. A
ring is called {\it reduced} if it has no nonzero nilpotent
elements. A weaker condition than ``reduced'' is defined by Lambek
in \cite{Lam}. A ring $R$ is said to be {\it symmetric} if for any
$a$, $b, c\in R$, $abc=0$ implies $acb=0$. Equivalently, $abc=0$
implies $bac=0$. An equivalent condition on a ring to be symmetric
is that whenever a product of any number of elements is zero, any
permutation of the factors still yields product zero.

In \cite{Mi}, a right ideal $I$ of $R$ is said to be {\it
reflexive} if $aRb \subseteq I$ implies $bRa \subseteq I$ for any
$a$, $b\in R$. $R$ is called a {\it  reflexive ring} if $aRb = 0$
implies $bRa = 0$ for any $a$, $b\in R$. And in  \cite{ZZG}, $R$
is said to be a {\it weakly reflexive ring} if $aRb = 0$ implies
$bRa\subseteq \mbox{nil}(R)$ for any $a$, $b\in R$. In \cite{KUH},
a ring $R$ is said to be {\it nil-reflexive} if $aRb\subseteq$
nil$(R)$  implies that $bRa\subseteq$ nil$(R)$  for any $a$, $b\in
R$. Let $R$ be a ring. In \cite{JAKL},  $R$ is called a {\it
reflexivity with maximal ideal axis ring}, in short, an {\it RM
ring} if $aMb = 0$ for a maximal ideal $M$ and for any $a$, $b\in
R$, then $bMa = 0$, similarly, $R$ has {\it reflexivity with
maximal ideal axis on idempotents}, simply, {\it RMI}, if $eMf =
0$ for any idempotents $e$, $f$ and a maximal ideal of $M$, then
$fMe = 0$. In \cite{KL1}, $R$ has {\it
reflexive-idempotents-property}, simply, {\it RIP}, if $eRf = 0$
for any idempotents $e$, $f$, then $fRe = 0$,  A left ideal $I$ is
called {\it idempotent reflexive} \cite{Ki} if $aRe\subseteq I$
implies $eRa\subseteq I$ for $a$, $e^2 = e\in R$. A ring $R$ is
called {\it idempotent reflexive} if $0$ is an idempotent
reflexive ideal. And Kim and Baik \cite{KB}, introduced the left
and right idempotent reflexive rings. A two sided ideal  $I$ of a
ring $R$ is called {\it right idempotent reflexive} if
$aRe\subseteq I$ implies $eRa\subseteq I$ for any $a$, $e^2 = e\in
R$. A ring $R$ is called {\it right idempotent reflexive} if  $0$
is the right idempotent reflexive ideal. Left idempotent reflexive
ideals and rings are defined similarly. If a ring $R$ is left and
right idempotent reflexive then it is called an idempotent
reflexive ring.

In this paper, motivated by these classes of types of reflexive
rings, we introduce left N-reflexive rings and right N-reflexive
rings. We prove that some results of reflexive rings can be
extended to the left N-reflexive rings for this general setting.
We investigate characterizations of left N-reflexive rings, and
that many families of left N-reflexive rings are presented.

In what follows, $\Bbb{Z}$ denotes the ring of integers and for a
positive integer $n$, $\Bbb{Z}_n$ is the ring of integers modulo
$n$. We write $M_n(R)$ for the ring of all $n\times n$ matrices,
$U(R)$,  nil$(R)$ will denote the  group of units and the set of
all nilpotent elements of $R$, $U_n(R)$  is the  ring of upper
triangular matrices over $R$ for a positive integer $n\geq 2$, and
$D_n(R)$ is the ring of all matrices in $U_n(R)$ having main
diagonal entries equal.

\section{N-reflexivity of rings} In this section, we introduce a class
of rings, so-called left N-reflexive rings and right N-reflexive
rings. These classes of rings generalize reflexive rings. We
investigate which properties of reflexive rings hold for the left
N-reflexive case. We supply an example to show that there are left
N-reflexive rings that are neither right N-reflexive nor reflexive
nor reversible. It is shown that the class of left N-reflexive
rings is closed under finite direct sums. We have an example to
show that homomorphic image of a left N-reflexive ring is not left
N-reflexive. Then, we determine under what conditions a
homomorphic image of a ring is left N-reflexive.

\noindent We now give our main definition.
\begin{df}{\rm Let $R$ be a ring and $I$ an ideal of $R$. $I$ is called {\it left N-reflexive} if for any
$a\in$ nil$(R)$, $b\in R$, being $aRb \subseteq I$ implies $bRa
\subseteq I$. The ring $R$ is called {\it left N-reflexive} if the
zero ideal is left N-reflexive. Similarly, $I$ is called {\it
right N-reflexive} if for any $a\in$ nil$(R)$, $b\in R$, being
$bRa \subseteq I$ implies $aRb \subseteq I$. The ring $R$ is
called {\it right N-reflexive} if the zero ideal is right
N-reflexive. The ring $R$ is called {\it N-reflexive} if it is
left and right N-reflexive.}
\end{df}

Clearly, every reflexive ring and every semiprime ring are
N-reflexive. There are left N-reflexive rings which are not
semiprime and there are left N-reflexive rings which are neither
reduced nor reversible.

Let $F$ be a field and  $R = F[x]$ be the polynomial ring  over
$F$ with $x$ an indeterminate and $\alpha : R\rightarrow R$ be a
homomorphism defined by $\alpha(f(x)) = f(0)$ where $f(0)$ is the
constant term of $f(x)$. Let $D^{\alpha}_2(R)$ denote skewtrivial
extension of $R$ by $R$ and $\alpha$. So $D^{\alpha}_2(R) =
\left\{\left(\begin{array}{cc}f(x)&g(x)\\0&f(x)\end{array}\right)\mid
f(x), g(x)\in R\right\}$ is the ring with componentwise addition
of matrices and multiplication: \begin{center}
$\left(\begin{array}{cc}f(x)&g(x)\\0&f(x)\end{array}\right)\left(\begin{array}{cc}h(x)&t(x)\\0&h(x)\end{array}\right)
= \left(\begin{array}{cc}f(x)h(x)&\alpha(f(x))t(x) +
g(x)h(x)\\0&f(x)h(x)\end{array}\right)$.\end{center} There are
left N-reflexive rings which are neither reflexive nor semiprime. The N-reflexive property of rings is not left-right symmetric.
\begin{ex}{\rm 
Let $D^{\alpha}_2(R)$ denote skewtrivial extension of $R$ by $R$
and $\alpha$ as mentioned above. Then by \cite[Example 3.5]{ZZG},
$D^{\alpha}_2(R)$ is not reflexive. We show that $D^{\alpha}_2(R)$
is left N-reflexive. Note that the set of all nilpotent elements
of $D^{\alpha}_2(R)$ is the set
$\left\{\left(\begin{array}{cc}0&f(x)\\0&0\end{array}\right)\mid
f(x)\in R\right\}$. Let $A =
\left(\begin{array}{cc}0&f(x)\\0&0\end{array}\right)$ be a
nilpotent in $D^{\alpha}_2(R)$ and  $B =
\left(\begin{array}{cc}h(x)&g(x)\\0&h(x)\end{array}\right)$ any
element in $D^{\alpha}_2(R)$. Assume that $AD^{\alpha}_2(R)B = 0$.
We may assume $f(x)\neq 0$. Then an easy calculation,
$AD^{\alpha}_2(R)B = 0$ reveals that $h(x) = 0$, and also
$BD^{\alpha}_2(R)A = 0$. Hence $D^{\alpha}_2(R) $ is left
N-reflexive. Next we show that $D^{\alpha}_2(R)$ is not right
N-reflexive. Let $A =
\left(\begin{array}{cc}0&f(x)\\0&0\end{array}\right)$ be a
nilpotent and $B =
\left(\begin{array}{cc}xh(x)&g(x)\\0&xh(x)\end{array}\right)$ any
element in $D^{\alpha}_2(R)$ with both $f(x)$ and $h(x)$ nonzero.
By definitions $BD^{\alpha}_2(R)A = 0$. Since $xh(x)f(x)$ is
nonzero, $AD^{\alpha}_2(R)B =
\left(\begin{array}{cc}0&f(x)r(x)xh(x)\\0&0\end{array}\right)$ is
nonzero for some nonzero $r(x)\in R$. So $D^{\alpha}_2(R)$ is not
right N-reflexive.

On the other hand, nil$(D^{\alpha}_2(R))$
is an ideal of $D^{\alpha}_2(R)$ and
$(\mbox{nil}(D^{\alpha}_2(R)))^2=0$ but nil$(D^{\alpha}_2(R))\neq
0$. Therefore $D^{\alpha}_2(R)$ is not semiprime.}
\end{ex}
\begin{prop}\label{ilk} Let $R$ be a left N-reflexive ring. Then for any
idempotent $e$ of $R$, $eRe$ is also left N-reflexive.
\end{prop}
\begin{proof} Let $eae\in eRe$ be a nilpotent and $ebe\in eRe$
arbitrary element with $eaeRebe = 0$. Then we have $ebeReae = 0$ since $R$ is left N-reflexive.
\end{proof}


\begin{exs}\label{ex}{\rm  (1) Let $F$ be a field and $R = M_2(F)$.
 In fact, $R$ is a simple ring, therefore prime. Let $A$, $B\in R$ with $ARB = 0$. Since $R$ is prime, $A = 0$ or
  $B = 0$. Hence $BRA = 0$. So
 $R$ is reflexive. Therefore $R$ is N-reflexive.\\
(2) Let $F$ be a field and consider the ring $R = D_3(F)$. Then
$R$ is neither left N-reflexive nor right N-reflexive. Let $A =
\left(\begin{array}{ccc}0&0&1\\0&0&1\\0&0&0\end{array}\right)$ be
nilpotent in $D_3(F)$ and $B =
\left(\begin{array}{ccc}0&1&1\\0&0&1\\0&0&0\end{array}\right)\in
D_3(F)$. Then $ARB = 0$. For $C =
\left(\begin{array}{ccc}1&1&1\\0&1&1\\0&0&1\end{array}\right)$,
$BCA =
\left(\begin{array}{ccc}0&0&1\\0&0&0\\0&0&0\end{array}\right)\neq
0$. Hence $R$ is not left N-reflexive.\\ Next we show that $R$ is
not right N-reflexive either. Now assume that $A =
\left(\begin{array}{ccc}0&1&1\\0&0&0\\0&0&0\end{array}\right)\in
\mbox{nil}(R)$ and $B =
\left(\begin{array}{ccc}0&1&1\\0&0&1\\0&0&0\end{array}\right)\in
R$. It is clear that $BRA = 0$.  For $C =
\left(\begin{array}{ccc}1&1&1\\0&1&1\\0&0&1\end{array}\right)\in
R$, we have $ACB\neq 0$. Hence $R$ is not right N-reflexive.}
\end{exs}
\begin{lem}\label{iso} N-reflexivity  of rings is preserved under isomorphisms.
\end{lem}

\begin{thm} Let $R$ be a ring. Assume that $M_n(R)$ is left N-reflexive. Then $R$ is left N-reflexive.
\end{thm}
\begin{proof}
Suppose that $M_n(R)$ is a left N-reflexive ring. Let $e_{ij}$
denote the matrix unit which $(i,j)$-entry is $1$ and the other
entries are $0$. Then $R\cong Re_{11}=e_{11}M_{n}(R)e_{11}$ is
N-reflexive by Proposition \ref{ilk} and Lemma \ref{iso}.
\end{proof}

\begin{prop}\label{reversible} Every reversible ring is left and right N-reflexive.
\end{prop}
\begin{proof} Clear by the definitions.
\end{proof}

The converse statement of Proposition \ref{reversible} may not be
true in general as shown below.
\begin{ex}{\rm By Examples \ref{ex}(1), $M_2(F)$ is both left and right N-reflexive. But it
is not reversible. }\end{ex}

\begin{thm} Let $R$ be a ring. Then the following are equivalent.
\begin{enumerate}
    \item $R$ is left N-reflexive.
    \item $IRJ=0$ implies $JRI=0$ for any ideal $I$ generated by a nilpotent element and
any nonempty subset $J$ of $R$.
    \item $IJ=0$ implies $JI=0$ for any ideal $I$ generated by a nilpotent element and any
    ideal $J$ of $R$.
\end{enumerate}
\end{thm}
\begin{proof} (1) $\Rightarrow$ (2) Assume that $R$ is left N-reflexive. Let $I=RaR$ with $a\in R$
nilpotent and $\emptyset\neq J\subseteq R$ such that $IRJ=0$. Then
for any $b\in J$, $aRb=0$. This implies that $bRa=0$, hence
$bR(RaR)=bRI=0$ for any $b\in J$. Thus $JRI=0$.\\
(2) $\Rightarrow$ (3) Let $I=RaR$ with $a\in R$ nilpotent and $J$
be an ideal of $R$ such that $IJ=0$. Then $J=RJ$, so $IRJ=0$. By
(2), $JRI=0$, thus $JI=0$.\\
(3) $\Rightarrow$ (1) Let $a\in R$ be nilpotent and $b\in R$ with
$aRb=0$. Then $(RaR)(RbR)=0$. By (3), $(RbR)(RaR)=0$. Hence
$bRa=0$. Therefore $R$ is left N-reflexive.
\end{proof}

For any element $a\in R$, $r_{R}(a)=\{b\in R\mid ab=0\}$ is called
the {\it right annihilator of $a$ in $R$}. The {\it left
annihilator of $a$ in $R$}  is defined similarly and denoted by
$l_{R}(a)$.
\begin{prop}
Let $R$ be a ring. Then $R$ is N-reflexive if and only if for any
nilpotent element $a$ of $R$, $r_{R}(aR)=l_{R}(Ra)$.
\end{prop}
\begin{proof}
For the necessity, let $x\in r_{R}(aR)$ for any nilpotent element
$a\in R$. We have $(aR)x=0$. The ring $R$ being N-reflexive
implies
$xRa=0$. So $x\in l_{R}(Ra)$. It can be similarly showed that $l_{R}(Ra)\subseteq r_{R}(aR)$.\\
\noindent For the sufficiency, let $a\in$ nil$(R)$ and $b\in R$
with $aRb=0$. Then $b\in r_{R}(aR)$. By hypothesis, $b\in
l_{R}(Ra)$, and so $bRa=0$. Thus $R$ is N-reflexive.
\end{proof}

For a field $F$, $D_3(F)$ is neither left N-reflexive nor right
N-reflexive. Subrings of left N-reflexive rings or right
N-reflexive rings need not be left N-reflexive  or right
N-reflexive, respectively. But there are some subrings of $D_3(F)$
that are left N-reflexive or right N-reflexive as shown below.

\begin{prop}\label{saturday} Let $R$ be a reduced ring (i.e., it has no nonzero nilpotent elements). Then the following hold.
\begin{enumerate}
\item Let $S =
\left\{\left(\begin{array}{ccc}a&b&c\\0&a&0\\0&0&a\end{array}\right)\mid
a, b, c\in R\right\}$ be  a subring of $D_3(R)$. Then $S$ is
N-reflexive.
\item Let $S =
\left\{\left(\begin{array}{ccc}a&0&c\\0&a&b\\0&0&a\end{array}\right)\mid
a, b, c\in R\right\}$ be  a subring of $D_3(R)$. Then $S$ is
N-reflexive.
\end{enumerate}
\end{prop}
\begin{proof} (1) Let $A = \left(\begin{array}{ccc}0&b&c\\0&0&0\\0&0&0\end{array}\right)\in S$ be
any nonzero nilpotent element and $B =
\left(\begin{array}{ccc}u&v&t\\0&u&0\\0&0&u\end{array}\right)\in
S$. Assume that $ASB = 0$. This implies $AB=0$, and so $bu=0$ and
$cu=0$. For any
$C=\left(\begin{array}{ccc}x&y&z\\0&x&0\\0&0&x\end{array}\right)\in
S$,
$BCA=\left(\begin{array}{ccc}0&uxb&uxc\\0&0&0\\0&0&0\end{array}\right)$.
Being $bu=cu=0$ implies $(uxb)^2=(uxc)^2=0$. Since $R$ is reduced,
$uxb=uxc=0$. Then $BCA=0$. Hence $BSA = 0$. Thus $R$ is left
N-reflexive. A similar proof
implicates that $S$ is right N-reflexive.\\
(2) By \cite[Proposition 2.2]{ZZG}.
\end{proof}

The condition $R$ being reduced in Proposition \ref{saturday} is
not superfluous as the following example shows.

\begin{ex}\label{örn}{\rm Let $F$ be a field and $R = F<a, b>$ be the free algebra with noncommuting indeterminates $a$,
$b$ over $F$. Let $I$ be the ideal of $R$ generated by $aRb$ and
$a^2$. Consider the ring $\overline {R} = R/I$. Let $\overline a$,
$\overline b\in \overline R$. Then $\overline a\overline
R\overline b = 0$. But $\overline b\overline R\overline a\neq 0$
since $ba\notin I$. Note that $\overline R$ is not reduced.
Consider the ring $S=\left\{\left(\begin{array}{ccc} \overline
x&\overline
y&\overline z\\ \overline 0&\overline x&\overline 0\\
\overline0&\overline 0&\overline x\end{array}\right)\mid \overline
x, \overline y, \overline z\in \overline R\right\}$. Let
$A=\left(\begin{array}{ccc} \overline a&\overline
1&\overline 1\\ \overline 0&\overline a&\overline 0\\
\overline0&\overline 0&\overline a\end{array}\right) \in nil(S) $ and
$B=\left(\begin{array}{ccc} \overline 0&\overline
b&\overline b\\ \overline 0&\overline 0&\overline 0\\
\overline0&\overline 0&\overline 0\end{array}\right) \in S$. Then
$ASB=0$ since $\overline a \overline R \overline b=0$. However
$BA\neq 0$ since $\overline b \overline a\neq 0$. Hence $S$ is not
left N-reflexive.}
\end{ex}

\begin{df}  Let $R$ be a ring. We call that $R$ is  {\it left  N-reversible} if for any nilpotent
$a\in R$ and $b\in R$, $ab = 0$  implies $ba = 0$.
\end{df}

In \cite{MMZ},  a ring $R$ is called {\it nil-semicommutative} if
for every nilpotent $a$, $b\in R$, $ab = 0$ implies $aRb = 0.$

\begin{thm} If a ring $R$ is N-reversible, then $R$ is nil-semicommutative  and N-reflexive.
\end{thm}
\begin{proof} Assume that $R$ is N-reversible. Let $a\in R$ be nilpotent and $b\in R$
with $aRb = 0$. Then $ab=0$. For any $r\in R$, being $abr=0$
implies $bra=0$, and so $bRa=0$. Hence $R$ is left N-reflexive. By
a similar discussion, $R$ is right N-reflexive. So $R$ is
N-reflexive. In order to see that $R$ is nil-semicommutative, let
$a, b\in R$ be nilpotent with $ab=0$. N-reversibility of $R$
implies $ba=0$, and so $bar=0$ for any $r\in R$. Again by the
N-reversibility of $R$, we have $arb=0$. Thus $aRb=0$.
\end{proof}

Note that in a subsequent paper, N-reversible rings will be
studied in detail by the present authors.

Let $R$ be a ring and $I$ an ideal of $R$. Recall by \cite{CKL},
$I$ is called {\it ideal-symmetric} if $ABC\subseteq I$ implies
$ACB\subseteq I$ for any ideals $A, B, C$ of $R$. In this vein, we
mention the following result.
\begin{prop} Let $R$ be a  ring and $I$ an ideal-symmetric ideal of $R$.
Then $R/I$ is an N-reflexive ring.
\end{prop}
\begin{proof} Let $\overline a \in R/I$ be nilpotent and $\overline b\in
R/I$ with $\overline a (R/I) \overline b=0$. Then $aRb\subseteq
I$. So $R(RaR)(RbR)\subseteq I$. By hypothesis,
$R(RbR)(RaR)\subseteq I$. Therefore $bRa\subseteq I$, and so
$\overline b(R/I)\overline a=0$. It means that $R/I$ is left
N-reflexive. Similarly, it can be shown that $R/I$ is also right
N-reflexive.
\end{proof}

Let $R$ be a ring and $I$ an ideal of $R$. In the short exact
sequence $0\rightarrow I\rightarrow R\rightarrow R/I\rightarrow
0$,  $I$ being N-reflexive (as a ring without identity) and $R/I$
being N-reflexive need not imply that $R$ is N-reflexive.

\begin{ex}{\rm Let $F$ be a field and consider the ring $R = D_3(F)$. Let
$I = \begin{pmatrix}0&F&F\\0&0&F\\0&0&0\end{pmatrix}$. Then $I$ is
 N-reflexive since  $I^3 = 0$. Also, $R/I$ is N-reflexive since
$R/I$ is isomorphic to $F$. However, by Examples \ref{ex} (2), $R$
is not N-reflexive. }\end{ex}

\begin{thm}\label{böl} Let $R$ be a ring and $I$ an ideal of $R$. If $I$ is
reduced as a ring (without identity) and $R/I$ is left
N-reflexive, then $R$ is left N-reflexive.
\end{thm}
\begin{proof} Let $a$ be nilpotent in $R$ and $b\in R$ with
$aRb=0$. Then $\overline a (R/I) \overline b=0$ and $\overline a$
is nilpotent in $R/I$. By hypothesis, $\overline b (R/I) \overline
a=0$. Hence $bRa\subseteq I$. Since $I$ is reduced and $bRa$ is
nil, $bRa=0$.
\end{proof}
The reduced condition  on the ideal $I$ in Theorem \ref{böl} is
not superfluous.
\begin{ex}{\rm Let $F$ be a field and $I = \left\{\left(\begin{array}{ccc}0&a&b\\0&0&c\\0&0&0\end{array}\right)\mid a,b,c\in F\right\}$
denote the ideal of $D_3(F)$. Then by Examples \ref{ex}(2),
$D_3(F)$ is not left and right N-reflexive. The ring $D_3(F)/I$ is
isomorphic to $F$ and so it is left and right N-reflexive.  Note
that $I$ is not reduced since $I^3 = 0$. }\end{ex}

The class of left (or right) N-reflexive rings are not closed under homomorphic images.
\begin{ex} \label{örnek}{\rm Consider the rings $R$ and $\overline {R} = R/I$ in the Example \ref{örn} where $I$ is the ideal of $R$ generated by $aRb$ and $a^2$. Then $R$ is
reduced, hence left (and right) N-reflexive. Let $\overline a$,
$\overline b\in \overline R$. Then $\overline a\overline
R\overline b = 0$. But $\overline b\overline R\overline a\neq 0$
since $ba\notin I$. Hence $R/I$ is not left N-reflexive. }\end{ex}

Let $e$ be an idempotent in $R$. $e$ is called {\it left
semicentral} if $re = ere$ for all $r\in R$. $S_l(R)$ is the set
of all left semicentral elements. $e$ is called {\it right
semicentral} if $er = ere$ for all $r\in R$. $S_r(R)$ is the set
of all right semicentral elements of $R$. We use $B(R)$ for the
set of central idempotents of $R$. In \cite{BKP}, a ring $R$ is
called {\it left(right) principally quasi-Baer} (or simply, {\it
left(right) p.q.-Baer}) ring if the left(right) annihilator of a
principal right ideal of $R$ is generated by an idempotent.
\begin{thm}\label{semicentral} The following hold for a ring $R$.
\begin{enumerate}
    \item[(1)] If $R$ is right N-reflexive, then $S_l(R)=B(R)$.
    \item[(2)] If $R$ is left N-reflexive, then $S_r(R)=B(R)$.
\end{enumerate}
\end{thm}
\begin{proof} (1) Let $e\in S_l(R)$ and $a\in R$. Then
$(1-e)Re=0$. It follows that $(1-e)Re(a-ae)=(1-e)R(ea-eae)=0$.
Since $ea-eae$ is nilpotent and $R$ is right N-reflexive,
$(ea-eae)R(1-e)=0$. Hence $(ea-eae)(1-e)=0$. This implies
$ea-eae=0$. On the other hand, $(1-e)R(a-ea)e\subseteq (1-e)Re=0$.
Thus $(1-e)R(ae-eae)=0$, and so $(1-e)(ae-eae)=0$. Then
$ae-eae=0$. So we have $ea=ae$, i.e., $e\in B(R)$. Therefore
$S_l(R)\subseteq B(R)$. The reverse inclusion is obvious. \\
(2) Similar to the proof of (1).
\end{proof}

\begin{thm}\label{Baer} Let $R$ be a right p.q-Baer ring. Then the following
conditions are equivalent.
\begin{enumerate}
\item[(1)] $R$ is a semiprime ring.
\item[(2)] $S_l(R)=B(R)$.
\item[(3)] $R$ is a reflexive ring.
\item[(4)] $R$ is a right N-reflexive ring.
\end{enumerate}
\end{thm}
\begin{proof} (1) $\Leftrightarrow$ (2) By \cite[Proposition
1.17(i)]{BKP}.\\
(1) $\Leftrightarrow$ (3) By \cite[Proposition 3.15]{KL}.\\ (3)
$\Rightarrow$ (4) Clear by definitions.\\
(4) $\Rightarrow$ (2) By Theorem \ref{semicentral}(1).
\end{proof}

\begin{thm} Let $R$ be a left p.q-Baer ring. Then the following
conditions are equivalent.
\begin{enumerate}
\item[(1)] $R$ is a semiprime ring.
\item[(2)] $S_r(R)=B(R)$.
\item[(3)] $R$ is a reflexive ring.
\item[(4)] $R$ is a left N-reflexive ring.
\end{enumerate}
\end{thm}
\begin{proof} Similar to the proof of Theorem \ref{Baer}.
\end{proof}

\begin{prop}
Let $R$ be a ring and $e^2=e\in R$. Assume that $R$ is an N-reflexive. Then $aRe=0$ implies $ea=0$ for any nilpotent element $a$ of $R$.
\end{prop}
\begin{proof} Suppose that $aRe=0$ for any $a\in$ nil$(R)$.
Since $R$ is N-reflexive, $eRa=0$, and so $ea=0$.
\end{proof}

\noindent \textbf{\emph{Question:}} If  a ring $R$ is N-reflexive, then is $R$ a $2$-primal ring?\\
\noindent There is a $2$-primal ring which is not N-reflexive.
\begin{ex}{\rm
Consider the $2$ by $2$ upper triangular matrix ring $R=\left(
                                                   \begin{array}{cc}
                                                     \Bbb{Z}_{2} & \Bbb{Z}_{2} \\
                                                     0 & \Bbb{Z}_{2} \\
                                                   \end{array}
                                                 \right)$ over the field $\Bbb{Z}_{2}$ of
integers modulo $2$. For $A=\left(
                              \begin{array}{cc}
                                0 & 1 \\
                                0 & 0 \\
                              \end{array}
                            \right)\in$ nil$(R)$ and $B=\left(
                                                           \begin{array}{cc}
                                                             1 & 1 \\
                                                             0 & 0 \\
                                                           \end{array}
                                                         \right)\in R$,
we have $ARB=0$ but $BRA\neq 0$. But $R$ is $2$-primal by
\cite[Proposition 2.5]{BHL}. }\end{ex}
\begin{prop} Let $\{R_i\}_{i\in I}$ be a class of rings. Then $R=\prod\limits_{i\in I}
R_i$ is left N-reflexive if and only if $R_i$ is left N-reflexive
for each $i\in I$.
\end{prop}
\begin{proof}
Assume that $R=\prod\limits_{i\in I} R_i$ is left N-reflexive. By
Proposition \ref{ilk}, for each $i\in  I$, $R_i$ is left
N-reflexive. Conversely, let $a=(a_i)\in R$ be nilpotent and
$b=(b_i)\in R$ with $aRb=0$. Then $a_iR_ib_i=0$ for each $i\in I$.
Since each $a_i$ is nilpoent in $R_i$ for each $i\in I$, by
hypothesis, $b_iR_ia_i=0$ for every $i\in I$. Hence $bRa=0$. This
completes the proof.
\end{proof}

\section{Extensions of N-reflexive rings}

In this section, we study some kinds of extensions of N-reflexive
rings to start with, the Dorroh extension $D(R,\Bbb{Z})=\{(r,n)\mid r\in R, n\in \Bbb{Z}\}$ of a ring $R$ is a ring with operations
$(r_1,n_1)+(r_2,n_2)=(r_1+r_2, n_1+n_2)$ and $(r_1,n_1)(r_2,n_2)=(r_1r_2+n_1r_2+n_2r_1, n_1n_2)$, where $r_i\in R$ and $n_i\in \Bbb{Z}$
for $i=1,2$.
\begin{prop} A ring $R$ is left N-reflexive if and only if the Dorroh extension $D(R,\Bbb
Z)$ of $R$ is left N-reflexive.
\end{prop}
\begin{proof} Firstly, we note that nil$(D(R,\Bbb
Z))=\{(r,0)\mid r\in \mbox{nil}(R)\}$. For the necessity, let
$(a,b)\in D(R,\Bbb Z)$ and $(r,0)\in $ nil$(D(R,\Bbb Z))$ with
$(r,0)D(R,\Bbb Z)(a,b)=0$. Then $(r,0)(s,0)(a,b)=0$ for every
$s\in R$. Hence $rs(a+b1_R)=0$ for all $s\in R$, and so
$rR(a+b1_R)=0$. Since $R$ is left N-reflexive, $(a+b1_R)Rr=0$.
Thus $(a,b)(x,y)(r,0)=((a+b1_R)(x+y1_R)r,0)=0$ for any $(x,y)\in
D(R,\Bbb Z)$. For the sufficiency, let $s\in R$ and $r\in $
nil$(R)$ with $rRs=0$. We have $(r,0)\in \mbox{nil}(D(R,
\Bbb{Z}))$. This implies $(r,0)D(R,\Bbb Z)(s,0)=0$. By hypothesis,
$(s,0)D(R,\Bbb Z)(r,0)=0$. In particular, $(s,0)(x,0)(r,0)=0$ for
all $x\in R$. Therefore $sRr=0$. So $R$ is left N-reflexive.
\end{proof}

Let $R$ be a ring and $S$ be the subset of $R$ consisting of
central regular elements. Set $S^{-1}R=\{s^{-1}r\mid s\in S, r\in
R\}$. Then $S^{-1}R$ is a ring with an identity.

\begin{prop} For a ring $R$, $R[x]$ is left N-reflexive if and
only if $(S^{-1}R)[x]$ is left N-reflexive.
\end{prop}
\begin{proof} For the necessity, let $f(x) = \displaystyle\sum_{i=0}^m s^{-1}_ia_ix^i$ be nilpotent and
$g(x)=\displaystyle\sum_{i=0}^n t^{-1}_ib_ix^i\in (S^{-1}R)[x]$
satisfy $f(x)(S^{-1}R)[x]g(x)=0$. Let $s = s_0s_1 \dots s_m$ and
$t = t_0t_1t_2 \dots t_n$. Then $f_1(x) = sf(x)$ is nilpotent and
$g_1(x)=tg(x)\in R[x]$ and $f_1(x)R[x]g_1(x) = 0$. By hypothesis,
$g_1(x)R[x]f_1(x) = 0$. Then $g(x)(S^{-1}R)[x]f(x) = 0$. The
sufficiency is clear.
\end{proof}

\begin{cor} For a ring $R$, $R[x]$ is left N-reflexive if and
only if $R[x;x^{-1}]$ is left N-reflexive.
\end{cor}

According to \cite{Hi}, a ring $R$ is said to be {\it
quasi-Armendariz} if whenever $f(x)=\sum_{i=0}^m a_ix^i$ and
$g(x)=\sum_{j=0}^n b_jx^j\in R[x]$ satisfy $f(x)R[x]g(x)=0$, then
$a_iRb_j=0$ for each $i,j$.

The left N-reflexivity or right N-reflexivity and the quasi-Armendariz property of rings do
not imply each other.
\begin{exs}{\rm (1) Let $F$ be a field and consider the ring $R=\left(
                                                          \begin{array}{cc}
                                                            F & F \\
                                                            0 & F \\
                                                          \end{array}
                                                        \right)$. Then $R$ is quasi-Armendariz by \cite[Corollary 3.15]{Hi}.
However, $R$ is not left N-reflexive. For $A=\left(
                                         \begin{array}{cc}
                                           0 & 1 \\
                                           0 & 0 \\
                                         \end{array}
                                       \right)\in$ nil$(R)$ and $B=\left(
                                                                   \begin{array}{cc}
                                                                     1 & 1 \\
                                                                     0 & 0 \\
                                                                   \end{array}
                                                                 \right)\in R$, we have $ARB=0$ but $BA\neq 0$.\\
(2) Consider the ring $R=\left\{ \left(\begin{array}{cc} a & b \\0 & a \\\end{array}\right)\mid a, b\in \Bbb{Z}_{4} \right\}$. Since $R$ is commutative,  $R$ is N-reflexive. For $f(x)=\left(
         \begin{array}{cc}
           0 & 1 \\
           0 & 0 \\
         \end{array}
       \right)+\left(
                 \begin{array}{cc}
                   2 & 1 \\
                   0 & 2 \\
                 \end{array}
               \right)x$ and $g(x)=\left(
                                     \begin{array}{cc}
                                       0 & 1 \\
                                       0 & 0 \\
                                     \end{array}
                                   \right)+\left(
                                             \begin{array}{cc}
                                               2 & 3 \\
                                               0 & 2 \\
                                             \end{array}
\right)x\in R[x]$, we have $f(x)Rg(x)=0$, and so by \cite[Lemma 2.1]{Hi} $f(x)R[x]g(x)=0$, but
$\left(
   \begin{array}{cc}
     2 & 1 \\
     0 & 2 \\
   \end{array}
 \right)R\left(
           \begin{array}{cc}
             0 & 1 \\
             0 & 0 \\
           \end{array}
         \right)\neq 0$. Thus $R$ is not quasi-Armendariz.}
\end{exs}

\begin{prop}\label{pol} Let $R$ be a quasi-Armendariz ring. Assume that coefficients of any nilpotent polynomial in $R[x]$ are nilpotent in $R$. Then $R$ is left N-reflexive if and only if $R[x]$ is left N-reflexive.
\end{prop}
\begin{proof} Suppose that $R$ is left N-reflexive and $f(x)=\sum_{i=0}^m
a_ix^i$, $g(x)=\sum_{j=0}^n b_jx^j\in R[x]$ with $f(x)R[x]g(x)=0$
and $f(x)$ nilpotent. The ring $R$ being quasi-Armendariz implies
$a_iRb_j = 0$ for all $i$ and $j$, and $f(x)$ being nilpotent
gives rise to all $a_0$, $a_1$, $a_2$, $\cdots$, $a_m$ nilpotent.
By supposition $b_jRa_i=0$ for all $i$ and $j$. Therefore
$g(x)R[x]f(x) = 0$, and so $R[x]$ is left N-reflexive. Conversely,
assume that $R[x]$ is left N-reflexive. Let $a\in R$ be nilpotent
and $b\in R$ any element with $aRb = 0$. Then $aR[x]b = 0$. Hence
$bR[x]a = 0$. Thus $bRa = 0$ and $R$ is left N-reflexive.
\end{proof}
Note that in commutative case, the coefficients of any nilpotent
polynomial are nilpotent. However, this is not the case for
noncommutative rings in general. Therefore in Proposition
\ref{pol} the assumption ``coefficients of any nilpotent
polynomial in $R[x]$ are nilpotent in $R$" is not superfluous as
the following example shows.
\begin{ex} {\rm Let $S = M_n(R)$ for a ring  $R$.
Consider the polynomial $f(x) = e_{21} + (e_{11} - e_{22})x -
e_{12}x^2\in S[x]$, where the $e_{ij}$'s are the matrix units.
Then $f(x)^2 = 0$, but $e_{11} - e_{22}$ is not nilpotent.}
\end{ex}
\section{Applications}
In this section, we study some subrings of full matrix rings
whether or not they are left or right N-reflexive rings.

 {\bf The
rings $H_{(s,t)}(R)$ :} Let $R$ be a ring and  $s, t$ be in the
center of $R$. Let\begin{center} $H_{(s,t)}(R) = \left
\{\begin{pmatrix}a&0&0\\c&d&e\\0&0&f
\end{pmatrix}\in M_3(R)\mid a, c, d, e, f\in R, a - d = sc, d - f = te\right \}$.\end{center}
Then $H_{(s,t)}(R)$ is a subring of $M_3(R)$. Note that any
element $A$ of $H_{(s,t)}(R)$ has the form
$\begin{pmatrix}sc+te+f&0&0\\c&te+f&e\\0&0&f\end{pmatrix}$.
\begin{lem}\label{nil} Let $R$ be a ring, and let $s, t$ be in the center of $R$. Then the set of all nilpotent elements of $H_{(s, t)}(R)$ is \begin{center} nil$(H_{(s, t)}(R)) = \left
\{\begin{pmatrix}a&0&0\\c&d&e\\0&0&f\end{pmatrix}\in H_{(s,
t)}(R)\mid a, d, f\in  nil(R), c, e\in R\right \}$.\end{center}
\end{lem}
\begin{proof} Let $A = \begin{pmatrix}a&0&0\\c&d&e\\0&0&f\end{pmatrix}\in$ nil$(H_{(s, t)}(R))$
be nilpotent. There exists a positive
integer $n$ such that $A^n = 0$. Then $a^n = d^n = f^n = 0$.
Conversely assume that $a^n = 0$, $d^m = 0$ and $f^k= 0$ for some
positive integers $n,m,k$. Let $p = max\{n, m, k\}$. Then $A^{2p}
= 0$.
\end{proof}
\begin{thm} The following hold for a ring $R$.
\begin{enumerate}
\item[(1)] If $R$ is a reduced ring, then $H_{(0, 0)}(R)$ is
N-reflexive but not reduced.
\item[(2)] If $R$ is reduced, then  $H_{(1, 0)}(R)$ is  N-reflexive but not reduced.
\item[(3)] If $R$ is reduced, then $H_{(0, 1)}(R)$ is N-reflexive but not reduced.
\item[(4)] $R$ is reduced if and only if $H_{(1, 1)}(R)$ is reduced.\end{enumerate}
\end{thm}
\begin{proof} (1) Let $A = \begin{pmatrix}a&0&0\\c&a&e\\0&0&a\end{pmatrix}\in$
nil$(H_{(0, 0)}(R))$ be nilpotent. By Lemma \ref{nil}, $a$ is
nilpotent. By assumption, $a = 0$. Let $B =
\begin{pmatrix}k&0&0\\l&k&n\\0&0&k\end{pmatrix}\in H_{(0,
0)}(R)$  with $AH_{(0, 0)}(R)B = 0$. $AB = 0$ implies $ck = 0$ and
and $ek = 0$. For any $X =
\begin{pmatrix}x&0&0\\y&x&u\\0&0&x\end{pmatrix}\in H_{(0,
0)}(R)$, $AXB =
\begin{pmatrix}0&0&0\\cxk&0&exk\\0&0&0\end{pmatrix} = 0$. Then $cxk = 0$ and $exk = 0$ for all $x\in R$.
The ring $R$ being reduced implies $kxc = 0$ and $kxe = 0$ for all
$x\in R$. Then $BXA =
\begin{pmatrix}0&0&0\\kxc&0&kxe\\0&0&0\end{pmatrix} = 0$ for all
$X\in H_{(0, 0)}(R)$. Hence $H_{(0, 0)}(R)$ is left N-reflexive. A
similar discussion reveals that $H_{(0, 0)}(R)$ is also right
N-reflexive. Note that being $R$ reduced does not imply $H_{(0,
0)}(R)$ is reduced because $A =
\begin{pmatrix}0&0&0\\1&0&1\\0&0&0\end{pmatrix}\in
H_{(0, 0)}(R)$ is a nonzero nilpotent element.\\
(2) Let $A =\begin{pmatrix} 0&0&0\\0&0&e\\0&0&0\end{pmatrix}\in $
nil$(H_{(1, 0)}(R))$ and  $B =
\begin{pmatrix} f+c&0&0\\c&f&d\\0&0&f\end{pmatrix}\in H_{(1,
0)}(R)$ with $AH_{(1, 0)}(R)B=0$. For any $C=\begin{pmatrix}
m+n&0&0\\n&m&u\\0&0&m\end{pmatrix}\in H_{(1, 0)}(R)$, $ACB=0$.
Then $emf=0$ and $fme=0$. This implies $BCA=0$. Therefore $H_{(1,
0)}(R)$ is left N-reflexive. Similarly, $H_{(1, 0)}(R)$ is also
right N-reflexive.
\\(3) Let $A =
\begin{pmatrix}0&0&0\\c&0&0\\0&0&0\end{pmatrix}\in
$ nil$(H_{(0, 1)}(R))$ and  $B =
\begin{pmatrix} e+f&0&0\\a&e+f&e\\0&0&f\end{pmatrix}\in H_{(0,
1)}(R)$ with $AH_{(0, 1)}(R)B=0$. For any $C=\begin{pmatrix}
m+n&0&0\\k&m+n&m\\0&0&n\end{pmatrix}\in H_{(0, 1)}(R)$, $ACB=0$.
Then $c(m+n)(e+f)=0$ and $(e+f)(m+n)c=0$. This implies $BCA=0$.
Therefore $H_{(0, 1)}(R)$ is left N-reflexive. Similarly, $H_{(0,
1)}(R)$ is also right N-reflexive.\\(4) Let $A =
\begin{pmatrix}c+e+f&0&0\\c&e+f&e\\0&0&f\end{pmatrix}\in
$ nil$(H_{(1, 1)}(R))$ be nilpotent. Then $f$ is nilpotent and so
$f = 0$. In turn, it implies $e = c = 0$. Hence $A = 0$.
Conversely, assume that $H_{(1, 1)}(R)$ is reduced. Let $a\in R$
with $a^n = 0$. Let $A =
\begin{pmatrix}a&0&0\\0&a&0\\0&0&a\end{pmatrix}\in H_{(1, 1)}(R)$.
Then $A$ is nilpotent. By assumption $a = 0$.
\end{proof}
\section{Generalizations, Examples and Applications}
 In this section, we introduce left N-right idempotent reflexive rings and right N-left idempotent
 reflexive rings generalize reflexive idempotent rings in Kwak and Lee \cite{KL}, and Kim
  \cite{Ki}, Kim and Baik in \cite{KB}. An ideal $I$ of a ring $R$
  is called {\it idempotent reflexive} if $aRe\subseteq I$ implies
  $eRa\subseteq I$ for any $a\in R$ and $e^2 = e\in R$. A ring $R$
  is said to be {\it idempotent reflexive} if the ideal 0 is idempotent reflexive.
   In \cite{KL1}, a ring $R$ is called to have the {\it reflexive-idempotents-property}
   if $R$ satisfies the property that $eRf = 0$ implies $fRe = 0$ for any  idempotents $e$ and $f$ of $R$. We introduce following some classes of rings to produce counter examples related to left N-reflexive rings. These classes of rings will be studied in detail in a subsequent paper by authors.
 \begin{df}{\rm Let $I$ be an ideal of a ring $R$. Then $I$ is called
  {\it left N-right idempotent reflexive} if being
 $aRe\subseteq I$ implies $eRa\subseteq I$ for any nilpotent $a\in R$ and $e^2 = e\in R$. A ring $R$ is called
 {\it left N-right idempotent reflexive} if  $0$ is a left N-right idempotent reflexive ideal.
 Left N-right idempotent reflexive ideals and rings are defined similarly. If a ring $R$ is left N-right
  idempotent reflexive  and right N-left idempotent reflexive, then it is called an {\it N-idempotent reflexive ring}.}
\end{df}

Every left N-reflexive ring is a left N-right idempotent reflexive
ring. But there are left N-right idempotent reflexive rings that
are not left N-reflexive.

\begin{exs}{\rm (1) Let $F$ be a field and $A = F\textless X, Y\textgreater$ denote the free algebra generated by noncommuting indeterminates $X$ and $Y$ over $F$.
Let $I$ denote the ideal generated by $YX$. Let $R = A/I$ and $x =
X + I$ and $y = Y + I\in R$. It is proved in \cite[Example 5]{Ki}
that $R$ is abelian and so $R$ has reflexive-idempotents-property
but not reflexive by showing that $xRy\neq 0$ and $yRx = 0$.
Moreover, $xyRx = 0$ and $xRxy \neq 0$. This also shows  that $R$
is not left N-reflexive since $xy$ is nilpotent in $R$.

(2) Let $F$ be a field and $A = F\textless X, Y\textgreater$
denote the free algebra generated by noncommuting indeterminates
$X$ and $Y$ over $F$. Let $I$ denote the ideal generated by $X^3$,
$Y^3$, $XY$, $YX^2$, $Y^2X$ in $A$. Let $R = A/I$ and  $x = X + I$
and $y = Y + I\in R$. Then in $R$, $x^3 = 0$, $y^3 = 0$, $xy = 0$,
$yx^2 = 0$, $y^2x = 0$. In \cite[Example 2.3]{JAKL}, $xRy = 0$,
$yRx\neq 0$ and idempotents in $R$ are 0 and 1. Hence for any
$r\in$ nil$(R)$ and $e^2 = e\in R$, $rRe = 0$ implies $eRr = 0$.
Thus $R$ is left N-right idempotent reflexive. We show that $R$ is
not a left N-reflexive ring. Since any $r \in R$ has the form $r =
k_0 + k_1x + k_2x^2 + k_3y + k_4y^2 + k_5yx$ and $x$ is nilpotent,
as noted above, $xRy = 0$. However, $yRx\neq 0$ since $yx \neq 0$.
Thus $R$ is not left N-reflexive.

(3) Let $F$ be a field of characteristic zero and $A = F\textless
X, Y, Z\textgreater$ denote the free algebra generated by
noncommuting indeterminates $X$, $Y$ and $Z$ over $F$. Let $I$
denote the ideal generated by $XAY$ and $X^2 - X$. Let $R = A/I$
and $x = X + I$, $y = Y + I$ and $z = Z + I\in R$. Then in $R$,
$xRy = 0$ and $x^2 = x$. $xy = 0$ and $yx$ is nilpotent and $x$ is
idempotent and $xRyx = 0$. But $yxRx \neq 0$. Hence $R$ is not
right N-left idempotent reflexive. In \cite[Example 3.3]{KL}, it
is shown that $R$ is right idempotent reflexive. }
\end{exs}


\begin{thebibliography}{99}
\bibitem{JAKL} A. M. Abdul-Jabbar, C. A. K. Ahmed, T. K. Kwak and
Y. Lee, {\it Reflexivity with maximal ideal axes}, Comm. Algebra,
45(10)(2017), 4348-4361.

\bibitem{BHL} G. F. Birkenmeier, H. E. Heatherly and E. K. Lee, {\it Completely prime ideals and associated radicals}, Ohio State-Denison Conference 1992, edited by S. K.
Jain and T. K. Rizvi, World Scientific, Singapore-New
Jersey-London-Hong Kong, (1993), 102-129.

\bibitem{BKP} G. F. Birkenmeier, J. Y. Kim and J. K. Park, {\it Principally quasi-Baer rings},
Comm. Algebra, 29(2)(2001), 639-660.

\bibitem{CKL} V. Camillo, T. K. Kwak and Y. Lee, {\it Ideal-symmetric and semiprime
rings}, Comm. Algebra, 41(12)(2013), 4504-4519.




\bibitem{Hi} Y. Hirano, {\it On annihilator ideals of a polynomial ring over a noncommutative
ring}, J. Pure Appl. Algebra, 168(1)(2002), 45-52.

\bibitem{Ki} J. Y. Kim, {\it Certain rings whose simple singular modules are GP-injective},
Proc. Japan Acad. Ser. A Math. Sci., 81(7)(2005), 125-128.

\bibitem{KB}   J. Y. Kim and  J. U. Baik, {\it On idempotent reflexive rings}, Kyungpook Math.
J.,
46(4)(2006), 597-601.


\bibitem{KUH} H. Kose, B. Ungor and A. Harmanci, {\it Nil-reflexive rings,}  Commun. Fac. Sci. Univ. Ank. S\'{e}r. A1 Math.
Stat., 65(1)(2016), 19-33.


\bibitem{KL} T. K. Kwak and Y. Lee, {\it Reflexive property of rings}, Comm. Algebra, 40(4)(2012), 1576-1594.

\bibitem{KL1} T. K. Kwak and Y. Lee, {\it Reflexive property on idempotents}, Bull. Korean Math. Soc., 50(4)(2013), 1957-1972.




\bibitem{Lam} J. Lambek, {\it On the representation of modules by sheaves of factor modules},
 Canad. Math. Bull., 14(1971), 359-368.




\bibitem{Mi} G. Mason, {\it Reflexive ideals}, Comm. Algebra, 9(1981), 1709-1724.

\bibitem{MMZ} R. Mohammadi, A. Moussavi and M. Zahiri, {\it On nil-semicommutative rings}, Int. Electron. J. Algebra, 11(2012), 20-37.



\bibitem{ZZG} L. Zhao, X. Zhu and Q. Gu, {\it Reflexive rings and their extensions,} Math. Slovaca, 63(3)(2013), 417-430.




\end{thebibliography}
\end{document}